\documentclass [a4paper,12pt]{amsart}
\usepackage{graphicx}
\usepackage[ansinew]{inputenc}    % accents i altres
\usepackage[all]{xy}
\usepackage{multicol}
\usepackage{wasysym}
\usepackage{amssymb,amscd}

\usepackage{enumitem}

\newtheorem{thm}{Theorem}[section]

\def\indgsv{\operatorname{Ind_{GSV}}}
\def\im{\operatorname{Im}}

\def\coker{\operatorname{coker}}

\def\derlog{\operatorname{Der(-log}}
\def\syzygy{\operatorname{syzygy}}

\def\C{\mathbb C}

\def\dim{\operatorname{dim}}

\def\coker{\operatorname{coker}}

\usepackage{graphicx}
\usepackage[ansinew]{inputenc}    % accents i altres
\usepackage[all]{xy}
\usepackage{amssymb,amscd}

\newtheorem{cor}[thm]{Corollary}
\newtheorem{teo}[thm]{Theorem}
\newtheorem{lem}[thm]{Lemma}
\newtheorem{prop}[thm]{Proposition}

\theoremstyle{definition}

\def\C{\mathbb C}

\def\dim{\operatorname{dim}}
\def\Tor{\operatorname{Tor}}

\def\coker{\operatorname{coker}}

\usepackage{color}

\thanks{The first author has been supported by FAPESP 2022/08662-1 and 2023/04460-8. The second author has been suported by Grant PID2021-124577NB-I00 funded by MCIN/AEI/ 10.13039/501100011033 and by ``ERDF A way of making Europe''. The third author has been pastially supported by FAPESP 2022/15458-1. The third and the fourth author have been partially supported by  FAPESP Grant 2019/07316-0.}

 \everymath{\displaystyle}
\begin{document}

\title[The Bruce-Roberts Numbers of 1-Forms on an ICIS]{The Bruce-Roberts Numbers of 1-Forms on an ICIS}

\author{B. K. Lima-Pereira, J.J. Nu\~no-Ballesteros, B. Or\'efice-Okamoto, J.N. Tomazella}

\address{ Instituto de Ci\^encias Matem\'aticas e de Computa\c c\~ao, Universidade de S\~ao Paulo, Caixa Postal 400, 13566-590, S\~ao Carlos, SP, Brazil} 
\email{barbaraklimapereira@icmc.usp.br}

\address{Departament de Matem\`atiques,
Universitat de Val\`encia, Campus de Burjassot, 46100 Burjassot
SPAIN. }
\email{Juan.Nuno@uv.es}

\address{Departamento de Matem\'atica, Universidade Federal de S\~ao Carlos, Caixa Postal 676,
	13560-905, S\~ao Carlos, SP, BRAZIL}
\email{brunaorefice@ufscar.br}

\address{Departamento de Matem\'atica, Universidade Federal de S\~ao Carlos, Caixa Postal 676,
	13560-905, S\~ao Carlos, SP, BRAZIL}

\email{jntomazella@ufscar.br}

\maketitle

\begin{abstract}
We relate the Bruce-Roberts numbers of a 1-form with respect to an ICIS to other invariants as the GSV-index, Tjurina and Milnor numbers. 
%In this work we were inspired by \cite{barbosa2024bruce} to extend our recent results about the Bruce-Roberts number of a function in \cite{lima2024bruce} for the new Bruce-Roberts numbers of a 1-form. More precisely we prove a relation for the Bruce-Roberts numbers of a 1-form with respect to an isolated complete intersection singularity (ICIS) using the known invariants of the variety and the 1-form. These results also extend the results in \cite{barbosa2024bruce} which they consider the particular case of the isolated hypersurface singularity (IHS).
\end{abstract}

\section{Introduction}

%An important invariant of the germ of an analytic function $f : (\C^{n}, 0) \to (\C, 0)$ is its Milnor number,
%$\mu(f)$, which is equal to $\dim_{\C}\mathcal{O}_{n}/Jf$, where $\mathcal{O}_{n}$ is the ring of germs of analytic functions on $(\C^n,0)$, and $Jf=\langle\partial f/\partial x_{i}\rangle$ is the Jacobian ideal of $f$.
%
 Let $(X, 0)\subset(\C^{n}, 0)$ be the germ of an analytic variety. We denote by $\Theta_{n}$  the $\mathcal O_n$-module of germs of vector fields on $(\C^n,0)$ and by $\Theta_{X}$, also denoted by $\derlog X)$,  the submodule of $\Theta_{n}$ consisting  of vector fields tangent to $X$ at its smooth points, that is, $$\Theta_{X}=\{\xi\in\Theta_{n}|\;dh(\xi)\in I_{X},\;\forall h\in I_{X}\},$$
where  $I_{X}\subset \mathcal{O}_{n}$ is the ideal that defines $(X,0).$  
Let $f:(\C^{n},0)\to(\C,0)$ be a function germ, the Bruce-Roberts number and the relative Bruce-Roberts number are, respectively, 
$$
\mu_{BR}(f,X)=\dim_{\C}\frac{\mathcal{O}_{n}}{df(\Theta_{X})} \, \, \textup{ and} \, \, \mu_{BR}^{-}(f,X)=\dim_{\C}\frac{\mathcal{O}_{n}}{df(\Theta_{X})+I_{X}},
$$ 
where $df(\Theta_{X})$ is the image of $\Theta_{X}$ by the differential of $f$, $df:\Theta_{n}\to\mathcal{O}_{n}$.
These numbers are defined in \cite{bruce1988critical} and may be considered as generalizations of the Milnor number of the function germ, because if $X=\C^{n}$ then $\Theta_{X}=\Theta_n$ and $df(\Theta_{X})=Jf$. 
Many authors present interesting results about these numbers \cite{ahmed2013invariants, bivia2022bruce,  bivia2024modules, bivia2020mixed, bruce1988critical,  dalbelo2020brasselet, grulha2009euler, kourliouros2021milnor, lima2022numero, lima2021relative, lima2024bruce, lima2023relative, nabeshima2021new, nuno2013bruce, nuno2020bruce}.  

% In general, the computation of both invariants is not easy since the submodule $\Theta_{X}$ is a complicated object and usually it requires the use of a computer algebra system like {\sc Singular} \cite{singular}. So, it is interesting to obtain formulas which give

We have proven the relations for $\mu_{BR}(f,X)$ and $\mu_{BR}^-(f,X)$ in terms of other known invariants when $(X,0)$ is an ICIS. More precisely we prove %The case where $(X,0)$ is an isolated hypersurface singularity IHS was considered previously in \cite{nunoballesteros oreficeokamoto limapereira tomazela, segundo artigo, Orefice}. In this paper we extend the formulas to the case that $(X,0)$ is an isolated complete intersection singularity ICIS. Our main results are
\begin{align}
\mu_{BR}^{-}(f,X)&=\mu(X\cap f^{-1}(0),0)+\mu(X,0)-\tau(X,0),\\
\mu_{BR}(f,X)&=\mu_{BR}^{-}(f,X)+\mu(f)-\dim_{\C}\frac{\mathcal{O}_{n}}{Jf+I_{X}}+\dim_{\C}\frac{I_{X}\cap Jf}{I_{X}Jf},
\end{align}
where $\mu(X,0)$ and $\mu(X\cap f^{-1}(0),0)$ are the Milnor number of the ICIS and $\tau(X,0)$ is the Tjurina number of the ICIS, see \cite{lima2024bruce}. %Also, we present a very direct proof of the famous equality between the Milnor and Tjurina numbers for a weighted homogeneous ICIS (see \cite{greuel mu=tau}). 

In \cite{barbosa2024bruce} it is defined the Bruce-Roberts number of a holomorphic 1-form $\omega\in (\C^{n},0)$ with respect a variety $(X,0)$, more precisely
$$
\mu_{BR}(\omega,X)=\dim_{\C}\frac{\mathcal{O}_{n}}{\omega(\Theta_{X})} \, \, \textup{ and} \, \, \mu_{BR}^{-}(\omega,X)=\dim_{\C}\frac{\mathcal{O}_{n}}{\omega(\Theta_{X})+I_{X}}.
$$

Although the proofs only require replacing the differential of the function with the 1-form in the proofs in \cite{lima2021relative, nuno2020bruce}, in \cite{barbosa2024bruce}, the analogous relationship between the Bruce-Roberts number of the 1-form and other invariants is presented.

%Although the proofs just require replacing the differential of the function by the 1-form in the proofs in the papers \cite{nuno2020bruce}, \cite{lima2021relative}.  The analogous results of \cite{lima2021relative} and \cite{nuno2020bruce}  relating the Bruce-Roberts to other invariants are presented in \cite{barbosa2024bruce}.  The proof of these results follows from the correspondent proves in \cite{lima2021relative} and \cite{nuno2020bruce} just replacing the differential of the function $f$ by the 1-formula $\omega$.

%Although the Bruce-Roberts number of a 1-form is more general than of a function germ the results about the Bruce-Roberts number in \cite{barbosa2024bruce} are obtained with the same technical that we use in our paper \cite{lima2021relative}, \cite{nuno2020bruce}, changing the partial derivatives of the function by the components of 1-form, see \cite[Theorem 5.2]{barbosa2024bruce} and \cite[Theorem 2.3]{lima2021relative}. 
%Moreover in \cite{barbosa2024bruce} they affirm that the $\mu_{BR}(f,X)$ is finite if, and only if, $f$ has an isolated singularity over $(X,0)$ and it is not true in general. Let $(X,0)\subset(\C^{4},0)$ be the hypersurface determined by $\phi(x,y,z,w)=x^3+x^2y^2+y^7+x^3z^2$ and $f(x,y,z,w)=x^2+y^2+z^2+w^2$, then 
%$$
%\mu_{BR}(f,X)=7,\; \mu_{BR}^{-}(f,X)=6 \textup{ and } \mu(X\cap f^{-1}(0))=\infty. 
%$$ 
%
%In fact we know this statement is true, when $(X,0)$ is an ICIS. In a general context should be true only in the isolated case.

In this paper we extend our previous result in \cite{lima2024bruce} for the Bruce-Roberts number of a 1-form. We denote the components of the 1-form $\omega$ by $\omega_{i}$ then in this case we have 
$\omega=\sum_{i=1}^{n}\omega_{i}dx_{i}$  and the Milnor number of 1-form $\omega$ is given by  $\mu(\omega)=\dim_{\C}\mathcal{O}_{n}/\langle\omega_{i}\rangle.$
%Another important property is that, like the Milnor number, the Bruce-Roberts numbers $\mu_{BR}(f,X)$ and $\mu_{BR}^{-}(f,X)$ may be calculated in terms of 
%the number of stratified critical points of a Morsification of $f$ with respect to the logarithmic stratification of $X$. This happens when 
%the logarithmic characteristic variety $LC(X)$ and its relative version $LC(X)^-$, respectively, are Cohen-Macaulay. In fact, the Cohen-Macaulayness of $LC(X)$ and $LC(X)^-$ implies the conservation of both numbers in any deformation of $f$. Many authors have recent papers about these issues \cite{tomazellaruas, carlesruas,  Nivaldo, Grego, nunoballesteros oreficeokamoto limapereira tomazela, segundo artigo, Orefice, tajima}.
%
%Here we show that if $(X,0)$ is any ICIS, then $LC(X)^-$ is Cohen-Macaulay and $LC(X)$ is also Cohen-Macaulay at any point not in $X\times\{0\}$. Again, this extends previous results of \cite{nunoballesteros oreficeokamoto limapereira tomazela,segundo artigo,  Orefice} when $(X,0)$ is an IHS and of \cite{bruce roberts} when $(X,0)$ is a weighted homogeneous ICIS. We remark that when $(X,0)$ has codimension $>1$, $LC(X)$ is not Cohen-Macaulay at any point in $X\times\{0\}$ (see \cite[Proposition 5.10]{bruce roberts}).
%
%As a byproduct of our process, we also prove that the Tjurina number $\tau(X,0)$ of an ICIS $(X,0)$ can be computed as
%\[
%\tau(X,0)=\dim_\C\frac{\Theta_X}{\Theta_X^T},
%\]
%where $\Theta_X^T$ is the submodule of $\Theta_X$ of trivial vector fields. This was proved in \cite{nunoballesteros oreficeokamoto limapereira tomazela, tajima} for IHS. 

\section{The relative Bruce-Roberts number of a 1-form}

Let $(X,0)$ be an ICIS determined by $\phi\colon (\C^{n},0)\to(\C^{k},0)$, that is,  $I_{X}=\langle\phi_{1},...,\phi_{k} \rangle$. In the same way in our previous work (\cite{lima2021relative, lima2024bruce, nuno2020bruce}) we consider the set of trivial vector fields $\Theta_{X}^{T}$ which are defined as
$$\Theta_{X}^{T}=\ker t\phi+I_X\Theta_n,$$
where $t\phi\colon\Theta_n\to\Theta(\phi)$ is the homomorphism  of $\mathcal{O}_{n}$-modules $\xi\mapsto d\phi\circ \xi$  and $\Theta(\phi)$ is the $\mathcal{O}_{n}$-module of germs of vector fields along $\phi$. We prove a characterization for $\Theta_{X}^{T}$.

\begin{prop}\cite[Proposition 2.1]{lima2024bruce}\label{prop:trivial} 
Let $(X,0)$ be the  ICIS determined by $\phi=(\phi_{1},...,\phi_{k}):(\C^{n},0)\to(\C^{k},0)$, then 
$$
\Theta_{X}^{T}=I_{k+1}\begin{pmatrix}
\tfrac{\partial}{\partial x_{1}}&\hdots&\tfrac{\partial}{\partial x_{n}}\vspace{0.1cm}\\
\tfrac{\partial\phi_{1}}{\partial x_{1}}&\hdots&\tfrac{\partial\phi_{1}}{\partial x_{n}}\\
\vdots&\ddots&\vdots\\
\tfrac{\partial\phi_{k}}{\partial x_{1}}&\hdots&\tfrac{\partial\phi_{k}}{\partial x_{n}}\\
\end{pmatrix}+\left\langle\phi_{i}\tfrac{\partial}{\partial x_{j}},\; i=1,...,k,\;j=1,...,n\right\rangle,
$$
where the first term in the right hand side is the submodule of $\Theta_X$ generated by the $(k+1)$-minors of the matrix.
\end{prop}

 As a consequence of the previous proposition, for any 1-form $\omega$, 
 \begin{equation}\label{dfthetaxt}
\omega (\Theta_{X}^{T})=I_{k+1}\begin{pmatrix}
\omega_{1}&\dots&\omega_{n}\\
\tfrac{\partial \phi_{1}}{\partial x_{1}}&\dots&\tfrac{\partial \phi_{1}}{\partial x_{n}}\\
\vdots&\ddots&\vdots\\
\tfrac{\partial \phi_{n}}{\partial x_{1}}&\dots&\tfrac{\partial \phi_{n}}{\partial x_{n}}\\
\end{pmatrix}+\langle\phi_{j}\omega_{i}\;i=1,...,n;\;j=1,...,k\rangle.
\end{equation} 

In \cite{lima2024bruce} we consider $(X,0)$ an ICIS and a $\mathcal{R}_{X}$-finitely determined function germ.  In this case $(X\cap f^{-1}(0),0)$ defines an ICIS. This condition is very important in our results. Considering a 1-form $\omega$ we use the GSV-index of 1-form over an ICIS $(X,0)$, $\indgsv(\omega,X,0)$. It could be defined algebraically as
$$
\indgsv(\omega,X,0)=\dim_{\C}\frac{\mathcal{O}_{n}}{I_{X}+I_{k+1}{\small \begin{pmatrix}\omega\\
d\phi
\end{pmatrix}}},
$$
with 
$$ 
\begin{pmatrix}
\omega\\
d\phi
\end{pmatrix}=
 \begin{pmatrix}\omega_{1}&\dots&\omega_{n}\\
\tfrac{\partial\phi_{1}}{\partial x_{1}}&\dots&\tfrac{\partial\phi_{1}}{\partial x_{1}}\\
\vdots &\ddots&\vdots\\
\tfrac{\partial\phi_{n}}{\partial x_{1}}&\dots&\tfrac{\partial\phi_{n}}{\partial x_{1}}\\
\end{pmatrix} $$

see \cite[Theorem 2]{ebeling2005indices}.

The next proposition is obtained in the same way  of \cite[Proposition 5.3]{lima2022numero}
\begin{prop}\label{relativo finito se e somente se indice finito}
Let $(X,0)$ be an ICIS determined by $\phi:(\C^{n},0)\to(\C^{k},0)$ and $\omega$ a 1-form. Then
$$
\mu_{BR}^{-}(\omega,X)<\infty \textup{ if, and only if, }\indgsv(\omega,X,0)<\infty.
$$
\end{prop}

\begin{teo}\label{relation}
Let $(X,0)\subset (\C^{n},0)$ be an ICIS and $\omega$ a 1-form such that $\mu_{BR}^{-}(\omega,X)<\infty$, then  $$\mu_{BR}^{-}(\omega,X)=\indgsv (\omega,X,0)+\tau(X,0).$$
\end{teo}

\begin{proof}
The proof of this theorem follows the same idea of \cite[Theorem 2.2]{lima2024bruce} changing the differential  of $f$, $df$, by the 1-form $\omega$.  The only part it is not immediate is to prove the matrix $\begin{pmatrix}
\omega\\
d\phi
\end{pmatrix}$ is a parameter matrix for the ring $R:=\mathcal{O}_{n}/I_{X}$ in the sense of \cite[p. 214]{buchsbaum1964generalized}. In fact, it follows from
 $$
 l(\coker\begin{pmatrix}
\omega\\
d\phi
\end{pmatrix}\otimes R)=\indgsv(\omega,X,0)<\infty\textup{ and }\dim R=n-k.
$$

\end{proof}

%We observe that the proof of the Theorem \ref{relation} is exactly the same that we present in \cite{lima2024bruce}, the unique difference is that we use that $\indgsv(\omega,X,0)$ is finite to conclude that 
%the matrix 
%$$
%\begin{pmatrix}
%\omega_{1}&\dots&\omega_{n}\\
%\tfrac{\partial \phi_{1}}{\partial x_{1}}&\dots&\tfrac{\partial \phi_{1}}{\partial x_{n}}\\
%\vdots&\ddots&\vdots\\
%\tfrac{\partial \phi_{n}}{\partial x_{1}}&\dots&\tfrac{\partial \phi_{n}}{\partial x_{n}}\\
%\end{pmatrix}
%$$
%is a parameter matrix for the ring $\mathcal{O}_{n}/I_{X}.$

%In the proof of the Theorem \ref{relation} we conclude the  dimension of $(\omega(\Theta_{X})+I_{X})/(\omega(\Theta_{X}^{T})+I_{X})$ as a $\C$-vector space does not depend on the 1-form $\omega$ such that $\indgsv(\omega,X,0)<\infty$.

\section{The Bruce-Roberts number of a 1-form}
%
% In \cite{barbosa2024bruce}, they prove if $(X,0)$ is an IHS and $\omega$ is a 1-form such that $\mu_{BR}(\omega,X)<\infty$ then 
% $$
% \mu_{BR}(\omega,X)=\mu(\omega)+\indgsv(\omega,X,0)-\tau(X,0).
% $$
%
%Although the Bruce-Roberts number of a 1-form is more general than of a function germ the proofs in \cite{barbosa2024bruce}, \cite{lima2021relative}, \cite{nuno2020bruce}, are very similar, in this way almost all results in \cite{barbosa2024bruce} are consequences of our previous one, \cite{lima2021relative}, \cite{nuno2020bruce}.

In this section, we present how these invariants are related when $(X,0)$ is an ICIS.

The proof of the following  result is the one of \cite[Proposition 2.8]{bivia2020mixed} replacing $df$ by $\omega$.

\begin{lem}\label{relativo finito se e somente se ind finito}
Let $(X,0)$ be an ICIS determined by $\phi:(\C^{n},0)\to(\C^{k},0)$, and $\omega$ a 1-form. If $\mu_{BR}(\omega,X)<\infty$ then $\indgsv(\omega,X,0)<\infty.$ 
\end{lem}
%\begin{proof}

%Suppose $\indgsv(\omega,X,0)=\infty$. By Nakayama's Lemma and the Hilbert Nullstelllensatz,  
%$$
%\dim V\left(I_{X}+I_{k+1}\begin{pmatrix}
%\omega\\
%d\phi
%\end{pmatrix},0 \right)\geq 1.
%$$
%
%Let $x\neq 0$, $x\in 
%V\left(I_{X}+I_{k+1}\begin{pmatrix}
%\omega\\
%d\phi
%\end{pmatrix},0 \right). 
%$ 
%Since $(X,0)$ is an ICIS and $x\in X$ not all minors of maximal order of the jacobian matrix of $\phi$ vanish at $x$, therefore $\omega(x)$ is a combination of $d\phi_{1}(x),\dots,d\phi_{k}(x)$.
%By hypothesis $\mu_{BR}^{-}(\omega,X)<\infty$, then $\omega|_{T_{x}X_{\alpha}}$ is  surjective, that is, it exists a vector field in $\xi\in\Theta_{X}$ such that $\omega(\xi)_{x}\neq 0$, but 
%$$
%\omega(\xi)_{x}=\sum_{j=1}^{k}\beta_{j}d\phi_{j}(\xi)_{x}=0.
%$$
%It is a contradiction. 
%
%The converse follows from the inclusion 
%$$
%I_{X}+I_{k+1}\begin{pmatrix}
%\omega\\
%d\phi
%\end{pmatrix}\subset \omega(\Theta_{X})+I_{X}
%$$  
%\end{proof}

\begin{prop} Let $(X,0)$ be an ICIS and $\omega$ a 1-form. Then
 $$\mu_{BR}(\omega,X)<\infty\textup{ if, and only if, }\dim_{\C}\mathcal{O}_{n}/\omega(\Theta_{X}^{T})<\infty.$$
\end{prop}  
\begin{proof}
Assume that $\mu_{BR}(\omega,X)<\infty$ but $\dim_{\C} \mathcal O_{n}/\omega(\Theta_{X}^{T})=\infty$. By Nakayama's Lemma and the Hilbert Nullstelllensatz, the variety 
$$
V(\omega(\Theta_{X}^{T}))=V \left(\langle\phi\omega_{i}\rangle+I_{k+1}\begin{pmatrix}\omega\\
d\phi
\end{pmatrix} \right)
$$ has dimension $\geq 1$.
Let $x\neq 0$ such that $x\in V(\omega(\Theta_{X}^{T}))$. In particular, 
 $$
 \phi(x)\omega_{i}(x)=0,\forall i=1,\dots,n.
 $$ 
Since $\mu(\omega)\leq\mu_{BR}(\omega,X)<\infty$, $(V(\langle\omega_{i}\rangle),0)\subset(\{0\},0)$. 
Hence  $x\in V\left(I_{X}+I_{k+1}\begin{pmatrix}
\omega\\
d\phi
\end{pmatrix}\right)$ and $\indgsv(\omega,X,0)=\infty$, which is a contradiction with the Lemma \ref{relativo finito se e somente se ind finito}.
  
The converse follows because $\omega(\Theta_{X}^{T})\subset \omega(\Theta_{X}).$
\end{proof}

\begin{prop}\label{ICIS dimensao on por dfthetaxt}
Let $(X,0)$ be an ICIS determined by $(\phi_{1},...,\phi_{k}):(\C^{n},0)\to(\C^{k},0)$ and $\omega$ a 1-form such that $\mu_{BR}(\omega,X)<\infty$, then 

$$\mu_{BR}(\omega,X)=\dim_{\C}\frac{\Theta(\phi)}{\langle\omega_{i}\rangle\Theta(\phi)+\langle\phi_{i}\frac{\partial}{\partial x_{j}}-\phi_{j}\frac{\partial}{\partial x_{i}}\rangle}+\indgsv(\omega,X,0)-\dim_{\C}\frac{\omega(\Theta_{X})}{\omega(\Theta_{X}^{T})}.$$
 \end{prop}

\begin{proof}
We consider the following  sequence
$$0\longrightarrow \frac{\Theta(\phi)}{\langle\omega_{i}\rangle\Theta(\phi)+\langle\phi_{i}\frac{\partial}{\partial x_{j}}-\phi_{j}\frac{\partial}{\partial x_{i}}\rangle}\stackrel{\alpha}{\longrightarrow}\frac{\mathcal{O}_{n}}{\omega(\Theta_{X}^{T})}\stackrel{\pi}{\longrightarrow}\frac{\mathcal{O}_{n}}{\omega(\Theta_{X}^{T})+\langle\phi_{1},...,\phi_{k}\rangle}\longrightarrow0$$
 with $\pi$ the projection and 
 $$
 \alpha((a_{1},...,a_{k})+\omega\Theta(\phi)+\langle\phi_{i}\partial/\partial x_{j}-\phi_{j}\partial/\partial x_{i}\rangle)=\sum_{i=1}^{k}a_{i}\phi_{i}+\omega(\Theta_{X}^{T}).
 $$
 Obviously, $\pi$ is an epimorphism and $\im\alpha=(I_{X}+\omega(\Theta_{X}^{T}))/\omega(\Theta_{X}^{T})=\ker\pi$. 
 
 To conclude the exactness of the sequence it only remains to see that $\alpha$ is a monomorphism. Since $\mu_{BR}(\omega,X)<\infty$, $\indgsv(\omega,X,0)<\infty$, which implies  
 \begin{equation}\label{eq:dim0}
 \dim\frac{\mathcal O_n}{\langle\phi_1,\dots,\phi_k\rangle+I_{k+1}\begin{pmatrix}\omega\\ d\phi\end{pmatrix}}=0.
 \end{equation}
This gives $\dim \mathcal O_n/I_{k+1}\begin{pmatrix}\omega\\
d\phi\end{pmatrix}\leq k$. Since $I_{k+1}\begin{pmatrix}\omega\\
d\phi\end{pmatrix}$ is the ideal generated by the maximal minors of a matrix of size $(k+1)\times n$, 
 $\mathcal O_n/I_{k+1}\begin{pmatrix}\omega\\
d\phi\end{pmatrix}$ is determinantal, and hence, Cohen-Macaulay of dimension $k$. Again by \eqref{eq:dim0} we conclude that 
 $$
 \phi_{1}+I_{k+1}\begin{pmatrix}
 \omega\\
 d\phi
 \end{pmatrix},...,\phi_{k}+I_{k+1}\begin{pmatrix}
 \omega\\
 d\phi
 \end{pmatrix}\textup{ is a regular sequence in }\mathcal{O}_{n}/I_{k+1}\begin{pmatrix}
 \omega\\
 d\phi
 \end{pmatrix}.$$ 

The proof to conclude that $\alpha$ is a monomorphism follows in the same way as in \cite[Proposition 4.2]{lima2024bruce}. 
\end{proof}

We need another characterization for the Tjurina number which will be a generalization for the IHS case obtained in \cite{barbosa2024bruce} considering a 1-form and IHS, and \cite{lima2024bruce} for any ICIS and an exact 1-form.% Moreover we observe the proof for this characterization in \cite[Theorem 5.2]{barbosa2024bruce} is very similar with we present in \cite[Theorem 2.3]{lima2021relative}.

\begin{lem}\cite[Lemma 4.3]{lima2024bruce}\label{caracterizacao dos campos triviais usando matrizes}
Let $(X,0)$ be an ICIS determined by $\phi=(\phi_{1},...,\phi_{k}):(\C^{n},0)\to(\C^{k},0)$. Then
 $\xi \in \Theta_{X}^{T}$ if and only if $t\phi(\xi)\in I_{X}t\phi(\Theta_{n})$.
\end{lem}
%\begin{proof}
%Let $\xi\in\Theta_{n}$ such that $t\phi(\xi)\in I_{X}t\phi(\Theta_{n})$. Then there exist 
%$\alpha_{ij}\in\mathcal{O}_{n}$, with $i=1,\dots,k$ and $j=1,\dots,n$ such that 
%$$
%t\phi(\xi)=\sum_{j=1}^{n}\sum_{i=1}^{k}\alpha_{ij}\phi_{i}t\phi\left(\frac{\partial}{\partial xj}\right)=t\phi\left(\sum_{j=1}^{n}\sum_{i=1}^{k}\alpha_{ij}\phi_{i}\frac{\partial}{\partial x_{j}}\right).
%$$
%Therefore we have that $\xi-\sum_{j=1}^{n}\sum_{i=1}^{k}\alpha_{ij}\phi_{i}\partial/\partial x_{j}\in\ker(t\phi)$, and $\xi\in\Theta_{X}^{T}$.
%The converse is immediate.
%\end{proof}

When $(X,0)$ is an ICIS and $f$ is a finitely $\mathcal{R}_{X}$-determined function germ we prove that for any $\xi \in\Theta _{X}$ such that $df(\xi)\in I_{X}$ is trivial. We extend this result for any ICIS and 1-form $\omega$ such that $\mu_{BR}(\omega,X)<\infty$, in order to prove it we need the following lemma.

\begin{lem}\label{icis e sequencia regular no modulo quociente im}
	Let  $(X,0)$ be the ICIS defined by $\phi=(\phi_{1},...,\phi_{k}):(\C^{n},0)\to(\C^{k},0)$ and $\omega$ a 1-form such that $\indgsv(\omega, X, 0)<\infty$. 
	
	Then the sequence $\phi_{1},...,\phi_{k}$ is regular in $\mathcal{O}_{n}^{k+1}/\im\begin{pmatrix}
		\omega\\
		d\phi
		\end{pmatrix}$.
\end{lem}

\begin{proof} The proof of follows in the same way as in \cite[Lemma 4.4]{lima2024bruce} considering the $\mathcal{O}_{n}$-module $M=\mathcal{O}_{n}^{k+1}/\im\begin{pmatrix}
 \omega\\
 d\phi\end{pmatrix}$, and the $\mathcal{O}_{X}$-module $M_{X}=\mathcal{O}_{n}^{k+1}\im\begin{pmatrix}
\omega\\
d\phi
\end{pmatrix}+I_{X}\mathcal{O}_{n}^{k+1}$.

% 
% We write  $M=\mathcal{O}_{n}^{k+1}/\im\begin{pmatrix}
% \omega\\
% d\phi\end{pmatrix}$. The sequence 
%$$\mathcal{O}_{n}^{k+1}\stackrel{{\tiny\begin{pmatrix}
% \omega\\
% d\phi\end{pmatrix}}}\longrightarrow\mathcal{O}_{n}^{k+1}\longrightarrow M\longrightarrow 0,$$ is a presentation of $M$, the $0$-th fitting ideal of $M$ is $F_{0}(M)=I_{k+1}\begin{pmatrix}
% \omega\\
% d\phi\end{pmatrix}$ and 
%$$\dim(M)=\dim\frac{\mathcal{O}_{n}}{\Ann(M)}=\dim\frac{\mathcal{O}_{n}}{I_{k+1}\begin{pmatrix}
% \omega\\
% d\phi\end{pmatrix}}=k.$$
%Hence, $M$ is a Cohen-Macaulay $\mathcal{O}_{n}$-module by \cite{buchsbaum1964generalized}.
%
%The $\mathcal O_{X}$-module $M_{X}=\mathcal{O}_{n}^{k+1}\im\begin{pmatrix}
%\omega\\
%d\phi
%\end{pmatrix}+I_{X}\mathcal{O}_{n}^{k+1}$ has the following presentation 
%$$\frac{\Theta_{n}}{I_{X}\Theta_{n}}\stackrel{{\tiny\begin{pmatrix}
% \omega\\
% d\phi\end{pmatrix}}}\longrightarrow\frac{\Theta(\phi)}{I_{X}\Theta(\phi)}\longrightarrow M_{X}\longrightarrow 0,$$
%and the $0$-th fitting ideal is $I_{k+1}\begin{pmatrix}
% \omega\\
% d\phi\end{pmatrix}$. Thus,
%$$\dim(M_{X'})=\dim\frac{\mathcal{O}_{X'}}{\Ann(M_{X'})}=\dim\frac{\mathcal{O}_{n}}{\langle\phi_{1},...,\phi_{k}\rangle+I_{k+1}\begin{pmatrix}
%\omega\\
%d\phi
%\end{pmatrix}}=0,
%$$ and therefore $\phi_{1},...,\phi_{k}$ is a regular sequence in $\mathcal{O}_{n}^{k+1}/I_{k+1}\begin{pmatrix}
%\omega\\
%d\phi
%\end{pmatrix}$ by \cite[Corollary B.8.3(2)]{greuel2007introduction}
\end{proof}

\begin{teo}\label{icis independe de f}
Let $(X,0)$ be an ICIS determined by $(\phi_{1},...,\phi_{k}):(\C^{n},0)\to(\C^{k},0)$ and $\omega$ a 1-form such that $\mu_{BR}(\omega,X)<\infty$.
\begin{enumerate}
\item\label{first part} If $\xi\in\Theta_{X}$ and $\omega(\xi)\in I_{X}$, then $\xi\in\Theta_{X}^{T}.$
\item The evaluation map $E:\Theta_{X}\to \omega(\Theta_{X})$ given by $E(\xi)=\omega(\xi)$ induces an isomorphism $$\overline{E}:\frac{\Theta_{X}}{\Theta_{X}^{T}}\to \frac{\omega(\Theta_{X})}{\omega(\Theta_{X}^{T
})}.$$
\end{enumerate}
\end{teo}
\begin{proof}
With the Lemma \ref{icis e sequencia regular no modulo quociente im} the proof of this theorem follows in the same way as in \cite[Theorem 4.5]{lima2024bruce} changing the 1-form $df$ by $\omega$. 
\end{proof}

\begin{cor}\label{tjurina para o bruce roberts}
Let $(X,0)$ be an ICIS and $\omega$ a 1-form such that $\mu_{BR}(\omega,X)<\infty$, then 
$$
\dim_{\C}\frac{\omega(\Theta_{X})}{\omega(\Theta_{X}^{T})}=\dim_{\C}\frac{\Theta_{X}}{\Theta_{X}^{T}}=\tau(X,0).
$$
\end{cor}
\begin{proof}
The first equality follows from Theorem \ref{icis independe de f} and the second one by \cite[Proposition 4.6]{lima2024bruce}. 

%As the dimension $\dim_{\C}\omega(\Theta_{X})/\omega(\Theta_{X}^{T})$ does not depend on $\omega$ we consider $f\in\mathcal{O}_{n}$ a $\mathcal{R}_{X}$-finitely determined function germ, and by \cite[Proposition 4.6]{lima2024bruce}
%$$
%\dim_{\C}\frac{\Theta_{X}}{\Theta_{X}^{T}}=\dim_{\C}\frac{\omega(\Theta_{X})}{\omega(\Theta_{X}^{T})}=\dim_{\C}\frac{df(\Theta_{X})}{df(\Theta_{X}^{T})}=\tau(X,0).
%$$
\end{proof}

Just like in \cite[p.15]{lima2024bruce} we obtain from the proof of Teorema \ref{icis independe de f} the isomorphisms
$$
{\frac{\omega(\Theta_{X})+I_{X}}{\omega(\Theta_{X})}}\approx\frac{{\Theta(\phi)}}{\langle \omega_{i}\rangle{\Theta(\phi)}+\syzygy(\phi_{1},...,\phi_{k})}\approx\frac{I_{X}}{\langle \omega_{i}\rangle I_{X}}.
$$
 
Therefore from Propositions \ref{ICIS dimensao on por dfthetaxt} and \ref{tjurina para o bruce roberts} we have %and Corollary \ref{I/JFI}
 \begin{equation}\label{igualdade que vou usar em baixo}
 \mu_{BR}(f,X)=\dim_{\C}I_{X}/\langle\omega_{i}\rangle I_{X}+\indgsv(\omega,X,0)-\tau(X,0).
 \end{equation}

\begin{teo}\label{relacao usando tor}
Let $(X,0)$ be an ICIS and $\omega$ a 1-form such that $\mu_{BR}(\omega,X)<\infty$, then 
\begin{align*}
\mu_{BR}(\omega,X)&=\mu(\omega)+\indgsv(\omega, X, 0)-\tau(X,0)-\dim_{\C}\frac{\mathcal{O}_{n}}{\langle\omega_{i}\rangle+I_{X}}+\dim_{\C}\frac{I_{X}\cap \langle\omega_{i}\rangle}{I_{X}\langle\omega_{i}\rangle}.\\
\mu_{BR}(\omega,X)&=\mu(\omega)+\mu_{BR}^{-}(\omega,X)-\dim_{\C}\frac{\mathcal{O}_{n}}{\langle\omega_{i}\rangle+I_{X}}+\dim_{\C}\frac{I_{X}\cap \langle\omega_{i}\rangle}{I_{X}\langle\omega_{i}\rangle}.
\end{align*}
\end{teo}
%\begin{proof} We consider the following exact sequence,
%$$0\longrightarrow \frac{I_{X}\cap Jf}{I_{X}Jf}\stackrel{i}\longrightarrow \frac{I_{X}}{I_{X}Jf} \stackrel{\pi}\longrightarrow \frac{I_{X}}{I_{X}\cap Jf}\longrightarrow 0, $$
%hence, \begin{align*}
%\dim_{\C}\frac{I_{X}}{I_{X}Jf}&=\dim_{\C}\frac{I_{X}\cap Jf}{I_{X}Jf}+\dim_{\C}\frac{I_{X}}{I_{X}\cap Jf}\\
%                              &=\dim_{\C}\frac{I_{X}\cap Jf}{I_{X}Jf}+\dim_{\C}\frac{I_{X}+Jf}{Jf}.
%                             %&=\dim_{\C}\Tor^{\mathcal{O}_{n}}_{1}\left(\frac{\mathcal{O}_{n}}{I_{X}},\frac{\mathcal{O}_{n}}{Jf}\right)+\mu(f)-\dim_{\C}\frac{\mathcal{O}_{n}}{I_{X}+Jf}.
%                             \end{align*}
%Using the equality (\ref{igualdade que vou usar em baixo}) we conclude the proof.
%\end{proof}
%
%In general, to calculate the dimension $\dim_{\C}(I_{X}\cap \langle\omega_{i}\rangle)/(I_{X}\langle\omega_{i}\rangle)$ is not easy. 
In order to improve the formula for $\mu_{BR}(\omega,X)$ in Theorem \ref{relacao usando tor}, we observe  
\[
\frac{I_{X}\cap \langle\omega_{i}\rangle}{I_{X}\langle\omega_{i}\rangle} \approx\Tor_{1}^{\mathcal{O}_{n}}\left(\frac{\mathcal{O}_{n}}{I_{X}},\frac{\mathcal{O}_{n}}{\langle\omega_{i}\rangle}\right), 
\]
(see \cite{de2013local}). 

Comparing the formula for $\mu_{BR}(\omega,X)$ in the previous theorem with the formula of \cite[Theorem 1]{barbosa2024bruce} in the IHS case, we get the following:

\begin{cor}\label{hipersuperfice}
Let $(X,0)\subset(\C^{n},0)$ be an IHS and $\omega$ a 1-form such that $\mu_{BR}(\omega,X)<\infty$. Then,
$$\dim_{\C}\Tor_{1}^{\mathcal{O}_{n}}\left(\frac{\mathcal{O}_{n}}{I_{X}},\frac{\mathcal{O}_{n}}{\langle\omega_{i}\rangle}\right)=\dim_{\C}\frac{\mathcal{O}_{n}}{I_{X}+\langle\omega_{i}\rangle}.$$
\end{cor}

\begin{cor}\label{caso k=2}Let $(X,0)$ be an ICIS of codimension $2$ and $\omega$  a 1-form such that $\mu_{BR}(\omega,X)<\infty$ then
$$\mu_{BR}(\omega,X)=\mu(f)+\indgsv(\omega,X,0)-\tau(X,0)+\dim_{\C}\frac{\mathcal{O}_{n}}{\langle\omega_{i}\rangle+\langle\phi_{1},\phi_{2}\rangle}.$$ 
\end{cor}
\begin{proof}
It follows immediately from our results in \cite[section 4.1]{lima2024bruce}
\end{proof}

%\section{References}
\bibliography{ref}
\bibliographystyle{abbrv}

\end{document}